\documentclass[12pt]{article}
\usepackage{fullpage}
\usepackage{amssymb, amsfonts, amsbsy, amsmath,latexsym, color,bbm,amsthm,mathrsfs}
\usepackage{graphicx}
\usepackage{hyperref}

\theoremstyle{definition}
\newtheorem{defn}{Definition}

\newtheorem{thm}[defn]{Theorem}
\newtheorem{prop}[defn]{Proposition}

\newtheorem{conj}[defn]{Conjecture}
\newtheorem{open}[defn]{Open Problem}

\newcommand{\cZ}{\mathcal{Z}}
\newcommand{\bR}{\mathbb{R}}

\newcommand{\bP}{\mathbb{P}}
\newcommand{\bZ}{\mathbb{Z}}
\newcommand{\bC}{\mathbb{C}}
\newcommand{\bF}{\mathbb{F}}
\newcommand{\bN}{\mathbb{N}}

\newcommand{\norm}[1]{\left\Vert#1\right\Vert}
\newcommand{\ip}[2]{\left<#1,#2\right>}
\newcommand{\absip}[2]{\left\vert\left<#1,#2\right>\right\vert}

\DeclareMathOperator*{\argmin}{\operatorname{arg\,min}}

\newcommand{\TF}{\operatorname{TF}}
\newcommand{\UN}{\operatorname{UN}}
\newcommand{\FUNTF}{\operatorname{FUNTF}}

\newcommand{\beq}{\begin{equation}}
\newcommand{\eeq}{\end{equation}}

\begin{document}

\title{Game of Sloanes:\\Best known packings in complex projective space}

\author{John Jasper\textsuperscript{a}, Emily J.~King\textsuperscript{b,c,$\star$}, and Dustin G.~Mixon\textsuperscript{d}\\
\textsuperscript{a} South Dakota State University, Brookings, SD, USA;\\
\textsuperscript{b} University of Bremen, Bremen, Germany;\\
\textsuperscript{c} Colorado State University, Fort Collins, CO, USA;\\
\textsuperscript{d} The Ohio State University, Columbus, OH, USA;\\
\textsuperscript{$\star$} corresponding author king@math.uni-bremen.de}


\maketitle

\begin{abstract}
It is often of interest to identify a given number of points in projective space such that the minimum distance between any two points is as large as possible.
Such configurations yield representations of data that are optimally robust to noise and erasures. 
The minimum distance of an optimal configuration not only depends on the number of points and the dimension of the projective space, but also on whether the space is real or complex.
For decades, Neil Sloane's online Table of Grassmannian Packings has been the go-to resource for putatively or provably optimal packings of points in real projective spaces.
Using a variety of numerical algorithms, we have created a similar table for complex projective spaces.
This paper surveys the relevant literature, explains some of the methods used to generate the table, presents some new putatively optimal packings, and invites the reader to competitively contribute improvements to this table.

\textbf{Keywords}: finite frames, Grassmannian frames, complex projective packings, coherence, maximal frame correlation, chordal distance, spherical codes, optimization
\end{abstract}

\section{Introduction}
Consider the problem of packing $n$ lines through the origin of $\bF^d$ such that they are maximally geometrically spread.
This problem has a number of applications in fields such as compressed sensing~\cite{BFMW13}, digital fingerprinting~\cite{MQKF13}, quantum state tomography~\cite{RBkSC04,FHS17}, and multiple description coding~\cite{StH03,MeDa14,Welch}.
Such configurations are also a cornerstone of discrete geometry~\cite{Toth65}.
It should not be surprising then that a number of results concerning these configurations have been independently rediscovered in a number of fields.
A closely related problem was initially proposed by a Dutch botanist almost one hundred years ago concerning the geometry of the surface of pollen~\cite{Tam30}.
Neil Sloane  and his coauthors indicated in their seminal paper~\cite{GrassPack} that their interest in optimal line packing started with a 1992 newsgroup post from an oncologist named Julian Rosenmann.
Namely, the doctor wanted to know the best way to separate laser beams going through a tumor.
These line configurations are often called Grassmannian frames since they form an optimal packing of the Grassmannian of lines.
Grassmannian frames have also been shown to be optimally robust to erasures~\cite{StH03,Bod07}, meaning that if one encodes data using a Grassmannian frame and then loses a coefficient, the reconstruction error will be the smallest amongst all systems of the same size.

For decades, Neil Sloane has hosted a website~\cite{Sloane} with the best known packings of lines (and higher-dimensional subspaces) in Euclidean space.
The main purpose of this paper is to announce a similar website~\cite{GameofSloanes} devoted to packing lines in complex space.
Users of the website will be able to download best-known packings as well as some of the Matlab code we used to discover the packings.
We also welcome the reader to join our so-called \textit{Game of Sloanes} by submitting packings that improve upon our putatively optimal packings.
We hope that by gamifying the search for optimal line packings, the community can eventually identify exact expressions for putatively optimal packings and even prove the optimality of these packings.

A second goal of this paper is to present a short survey of results concerning the problem of packing lines in $\bF^d$.
We begin in Section~\ref{sec:frame} by introducing the basics of frame theory.
Indeed, it is particularly convenient to express lines in terms of vectors, and it is also convenient to discuss these vectors in the language of frame theory.
In Section~\ref{sec:grass}, we explicitly define what is meant by ``maximally geometrically spread,'' and we identify various frame properties that such configurations necessarily exhibit.
Section~\ref{sec:opt} contains a number of bounds on how geometrically spread a given number of lines can be in a particular space.
Any packing which achieves equality in one of these bounds is necessarily an optimal packing, and this is the most common way to demonstrate optimality.
Still, there are other methods to show optimality, and we will discuss them as well.
Finally, in Section~\ref{sec:games}, we introduce the website~\cite{GameofSloanes}, we isolate an open problem, and we pose two conjectures.



\section{Frame Theory Basics}\label{sec:frame}
Throughout, $[n]$ denotes the set $\{1, \dots, n\}$, whereas $\bF$ can be taken to be either $\bR$ or $\bC$.
By slight abuse of notation, $\Phi$ will denote both a sequence $\left(\varphi_j\right)_{j=1}^n$ in $\bF^d$ and the $d \times n$ matrix whose $j$th column is $\varphi_j$ (the intended interpretation will be clear from context).
Further, we write $\Phi \in \UN(d,n,\bF)$ if $\norm{\varphi_j} =1$ for all $j \in [n]$.

Frames provide a fruitful generalization of orthonormal bases.
At the weakest level, they satisfy a loosening of Parseval's equality, but there are special classes of frames that behave even more like orthonormal bases.

\begin{defn}
Let $\Phi = \left(\varphi_j\right)_{j=1}^n \subset \bF^d$.  Then $\Phi$ is a \emph{frame} for $\bF^d$ if there exist \emph{frame bounds} $0  < A \leq B$ such that for every $x \in \bF^d$, it holds that
\beq\label{eq:framebd}
A \norm{x}^2 \leq \sum_{j=1}^n \absip{x}{\varphi_j}^2 \leq B \norm{x}^2.
\eeq
If one may select $A = B$ in~\eqref{eq:framebd}, then we call $\Phi$ a \emph{tight frame} and we write $\Phi \in \TF(d,n,\bF)$.
Finally, if $\Phi \in \UN(d,n,\bF) \cap \TF(d,n,\bF)$, we write $\Phi \in \FUNTF(d,n,\bF)$ and call $\Phi$ a \emph{finite unit norm tight frame (FUNTF)}.
\end{defn}

While general frames are interesting in infinite dimensional Hilbert spaces, in the finite dimensions we consider in this paper, a frame is simply a spanning set.
In particular, $n \geq d$ must hold in order for $\Phi$ to be a frame.
When a frame is tight, it enjoys a nice change-of-basis-type formula.
Namely, if $\Phi \in \TF(d,n\bF)$ with frame bound $A$, then for all $x \in \bF^d$, it holds that
\beq\label{eq:tight}
x = \sum_{j=1}^n \frac{1}{A} \ip{x}{\varphi_j} \varphi_j.
\eeq

Unit vectors may be thought of as points on a sphere, and some applications are concerned with the minimum geodesic distance of a given finite collection of points on the sphere~\cite{delsarte1977,musin2009}.

\begin{defn}
Let $S^{d-1}$ denote the unit sphere in $\mathbb{R}^d$.
For any $x,y\in S^{d-1}$, the \emph{spherical distance} between them is $\arccos(\ip{x}{y})$.
Every $\Phi =\left(\varphi_j\right)_{j=1}^n \in \UN(d,n,\bR)$ is called a \emph{spherical code}. 
If the set of inner products $\{ \ip{\varphi_j}{\varphi_k}\, : \, j \neq k\}$ only has two elements $\alpha$, $\beta$, we call $\Phi$ a \emph{spherical two-distance set}.
A spherical two-distance set with $\beta^2 = \alpha^2$ is called \emph{equiangular}.  More generally, a set  $\left(\varphi_j\right)_{j=1}^n \in \UN(d,n,\bF)$ is called equiangular if $\absip{\varphi_j}{\varphi_k}$ is constant for all $j \neq k$.
\end{defn}

\section{Grassmannian Frames}\label{sec:grass}

Orthonormal bases have the property that the lines they span are as geometrically spread apart as possible.
We commonly use \textit{coherence} to quantify this geometric spread.

\begin{defn}
Let $\Phi = \left(\varphi_j\right)_{j=1}^n \in \UN(d,n,\bF)$. The \emph{coherence} $\mu$ of $\Phi$ is defined to be
\[
\mu\left(\Phi\right) = \max_{1\leq  j < k\leq n} \absip{\varphi_j}{\varphi_k}.
\]
The \emph{angle set} of $\Phi$  is $\{\absip{\varphi_j}{\varphi_k}^2\,:\, j,k \in [n], \,j\neq k \}$.

We call $\Phi$ a \emph{Grassmannian frame} if
\beq\label{eq:Grassman}
\Phi \in \argmin\left\{ \mu(\Psi) \, : \, \Psi \in \UN(d,n,\bF) \right\}.
\eeq
\end{defn}

Coherence is also called \emph{maximal frame correlation} (cf.~\cite{StH03, BKo06} and contemporaneous publications) and more generally categorized as \emph{geometric spread} (as opposed to algebraic spread)~\cite{King15}.
If $n \leq d$, then there exist sets of $n$ orthonormal vectors in $\bF^d$, which have coherence $0$.
Thus the Grassmannian frames for $n \leq d$ are necessarily orthonormal vectors.
(Of course, these vectors fail to span when $n<d$.
While such packings are not \textit{frames}, we still call them \textit{Grassmannian frames} since they minimize coherence.
For the case where $n\geq d$, this choice of nomenclature is justified in Proposition~\ref{prop.grass frames are frames} below.)
It follows from a compactness argument that Grassmannian frames exist for all choices of $d$, $n$, and $\bF$.

The above definition is restricted to unit vectors for two reasons. 
Initially, using equal-norm vectors in applications ensures that no one vector is weighted more than any others.  
Further, from a more theoretical viewpoint, we are concerned with the lines spanned by the vectors.  
Thus, each of the unit vectors forms orthonormal basis for the line it spans.  
This viewpoint also explains the appearance of the absolute value in the definition of coherence; one may sign/phase any unit vector to get a different orthonormal basis for the same line.  
In particular, $\absip{\varphi_j}{\varphi_k}$ yields the cosine of the interior angle between the lines spanned by $\varphi_j$ and $\varphi_k$. 
A Grassmannian frame is thus a set of $n$ vectors which span lines in $\bF^d$ which have interior angles that are as large as possible given $d$, $n$, and $\bF$.  
In the real case, $\ip{\varphi_j}{\varphi_k}$ returns the cosine of the spherical distance between $\varphi_j,\varphi_k\in S^{d-1}$.
Solving the \emph{Tammes problem} (also known as \emph{Toth's problem}) amounts to minimizing the maximum of $\ip{\varphi_j}{\varphi_k}$ over $\UN(d,n,\bR)$~\cite{Tam30,Toth49}, with special attention paid to the classical case where $d=3$; note the subtle distinction from minimizing coherence.
Such configurations are also called \emph{optimal spherical codes} or \emph{spherical packings} (see, e.g,~\cite{ConwaySloane98}). 

Minimizing the coherence of vectors is equivalent to maximizing the \emph{chordal distance} between points in the Grassmannian of lines in $\bF^d$, hence the name, which seems to have first appeared in~\cite{StH03}.  
Since the Grassmannian of lines in $\bF^d$ is equivalent to projective space $\bF\bP^{d-1}$, Grassmannian frames are sometimes called \emph{(optimal) projective packings} (see, e.g.~\cite{FJM18}) or \emph{optimal projective codes} (see, e.g.~\cite{Glaz19}).  
In the mid-twentieth century, authors often referred to the distribution of unit vectors relative to the absolute value of their inner product as a problem in \emph{elliptic geometry} (see, e.g.,~\cite{Toth65,VaSei66}). 
Finally, one may view the problem of optimizing the coherence as maximizing the spherical distance between points in \emph{antipodal codes}, which are spherical codes in which the antipode of every point in the code must also reside in the code.

We have not yet shown that a Grassmannian frame is indeed a frame.  
In what follows, we slightly modify the proof from~\cite{FJM18} that treats the real case.
A similar, non-rigorous argument may be found in~\cite{HaCa17}.

\begin{prop} 
\label{prop.grass frames are frames}
If $\Phi = (\varphi_j)_{j=1}^n \subset \bF^d$ satisfies~\eqref{eq:Grassman} and $n \geq d$, then it is a frame.
\end{prop}

\begin{proof}
We must show that $\Phi$ is a spanning set for $\bF^d$.  
If $n = d$, then any $\Phi$ which satisfies~\eqref{eq:Grassman} is an orthonormal basis and trivially spans $\bF^d$.  
Hence we may assume without loss of generality that $n > d$ and $\mu=\mu(\Phi) > 0$. 

We prove the claim by way of contradiction and assume that $\Phi$ is not spanning.  
For each $j \in [n]$, we define 
\[
N(j; \Phi) = \left\{ k \in [n]: \absip{\varphi_j}{\varphi_k} = \mu \right\}.
\]
Since we assumed that $\Phi$ is not spanning, there is no $j \in [n]$ such that $(\varphi_k)_{k \in N(j; \Phi)}$ spans $\bF^d$. 
We will now modify each $\varphi_j$, starting with $j=1$ and iterating incrementally. 
Let $v_j$ be a unit vector in the orthogonal complement of $(\varphi_k)_{k \in N(j; \Phi)}$. 
We define for each $t \in (-1,1)$ the quantity
\[
\psi_j(t) = \frac{\varphi_j + t v_j}{\norm{\varphi_j + t v_j}}.
\]
We would like to show that there is a $t$ such that $\absip{\varphi_k}{\psi_j(t)}< \mu$ for all $k \in [n] \setminus \{j\}$ and then replace $\varphi_j$ with the corresponding $\psi_j(t)$.  
We update the sets $N(k; \Phi)$ for all applicable $k$ after replacing $\varphi_j$.  
We see that after the $j$th iteration, $N(j; \Phi)$ becomes the empty set and $j$ is removed from $N(k; \Phi)$ for any $k < j$.

We have not yet shown that such a $t$ even exists. Since $t \mapsto \psi_j(t)$ is continuous for $t \in (-1,1)$, for any sufficiently small $t$ and $k \notin N(j; \Phi)$, $\absip{\varphi_k}{\psi_j(t)}< \mu$ still holds. Let $t \neq 0$ be such that $\operatorname{Re}(t \ip{\varphi_j}{v_j}) \geq 0$.  Then
\[
\norm{\varphi_j + t v_j}^2 = \norm{\varphi_j}^2 + 2\operatorname{Re}(t\ip{\varphi_j}{v_j}) + t^2 \norm{v_j}^2 = 1 +2\operatorname{Re}(t\ip{\varphi_j}{v_j}) + t^2  \geq 1 + t^2.
\]
If $k \in N(j; \Phi)$, then for any $t \neq 0$ satisfying $\operatorname{Re}(t \ip{\varphi_j}{v_j}) \geq 0$, it holds that
\[
\absip{\varphi_k}{\psi_j(t)} = \frac{\absip{\varphi_k}{\varphi_j + tv_j} }{\norm{\varphi_j + t v_j}} =  \frac{\absip{\varphi_k}{\varphi_j} }{\norm{\varphi_j + t v_j}} \leq \frac{\mu}{\sqrt{1+t^2}} < \mu,
\]
as desired.
\end{proof}

Note that we actually proved a stronger statement concerning the structure of a Grassmannian frame:
If $\Phi = (\varphi_j)_{j=1}^n \subset \bF^d$ is a Grassmannian frame, then there exists a $j \in [n]$ such that $\bF^d$ is spanned by the vectors $\varphi_k$ satisfying $\absip{\varphi_k}{\varphi_j} = \mu(\Phi)$.

\section{Proving Optimality of a Packing}\label{sec:opt}
Given $\Phi \in \UN(d,n,\bF)$, the easiest way to demonstrate that $\Phi$ is a Grassmannian frame is by showing that $\mu(\Phi)$ achieves equality in one of four popular lower bounds.  
In order to formulate these bounds, we require the following definition. 

\begin{defn}
Fix $d$ and $\bF$. 
\emph{Gerzon's bound} $\cZ(d,\bF)$ is defined as
\[
 \cZ(d,\bF) = \left\{\begin{array}{cl}d^2 & \text{if } \bF = \bC \\  \frac{d(d+1)}{2} & \text{if } \bF = \bR. \end{array} \right.
\]
\end{defn}

\begin{thm} \label{thm:bounds}
Fix $d>1$, select $\Phi \in \UN(d,n,\bF)$, and set $\mu = \mu(\Phi)$.
Define $2m = \dim_\bR \bF$.
Then
\begin{align}
\label{eq:BC}
\mu(\Phi) 
&\geq  \frac{\cZ(n-d, \bF)}{n\left(1+m(n-d-1)\sqrt{m^{-1}+n-d}\right)-\cZ(n-d,\bF)}
&& \textrm{if $n > d$,}\\
\label{eq:Welch}
\mu(\Phi)
&\geq \sqrt{\frac{n-d}{d(n-1)}}
&& \textrm{if $n > d$,}\\
\label{eq:orth}
\mu(\Phi) 
&\geq \frac{1}{\sqrt{d}}
&&\textrm{if $n > \cZ(d,\bF)$,}\\
\label{eq:Lev}
\mu(\Phi) 
&\geq \sqrt{\frac{n(m+1)-d(md+1)}{(n-d)(md+1)}}
&& \textrm{if $n>\cZ(d,\bF)$.}
\end{align}
Equality holds in~\eqref{eq:Welch} if and only if $\Phi$ is an equiangular tight frame, which requires $n \leq \min\{\cZ(d,\bF), \cZ(n-d,\bF)\}$. 
Equality in~\eqref{eq:orth} requires $n \leq 2(\cZ(d,\bF)-1)$, and equality in~\eqref{eq:Lev} requires that $\Phi$ be tight with angle set $\{0,\mu^2\}$.
\end{thm}


\begin{figure}[ht!]
\centering
\includegraphics[width=0.8\textwidth]{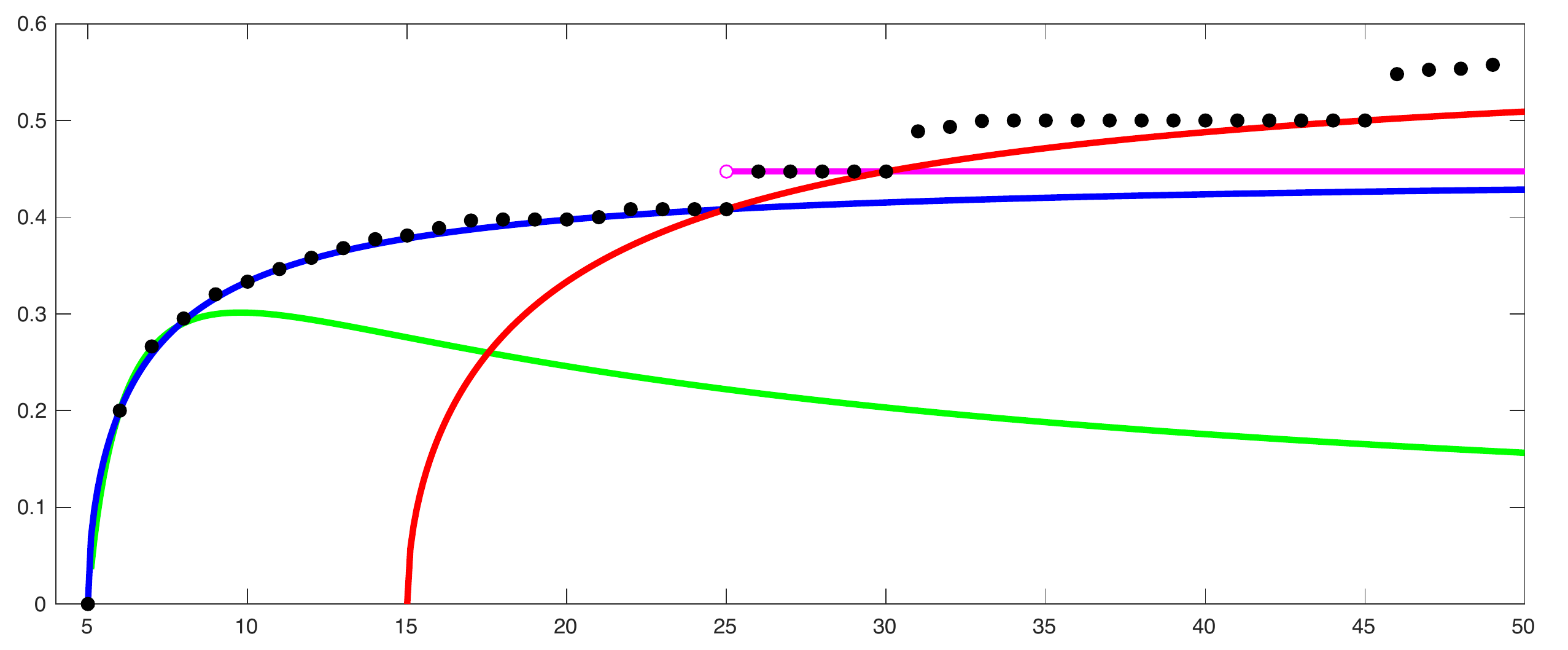}
\caption{The various bounds from Theorem~\ref{thm:bounds} in $\bC^5$ are visualized by smoothed colored lines.  
The best known packings~\cite{GameofSloanes} for each $n$ are denoted by black dots.  
The Bukh--Cox bound~\eqref{eq:BC} is in green, the Welch--Rankin bound~\eqref{eq:Welch} in blue, the orthoplex bound~\eqref{eq:orth} in pink, and the Levenstein bound~\eqref{eq:Lev} in red.}
\label{fig:bounds}
\end{figure}

See Figure~\ref{fig:bounds} for an illustration of Theorem~\ref{thm:bounds} in the case where $d=5$ and $\bF=\bC$.
The bound in~\eqref{eq:Welch} has a long history of being rediscovered in a number of contexts.  
It is known as the \emph{Welch bound}, the \emph{Welch--Rankin bound}, the \emph{first Rankin bound}, the \emph{simplex bound}, and the \emph{first Levenstein bound}. 
In 1955, Rankin proved certain bounds on spherical packings~\cite{Ran55}.  
The authors in~\cite{GrassPack} embed projections onto subspaces into a sphere in $\bF^{\cZ(d,\bF)-1}$ and then apply the results of Rankin to obtain both~\eqref{eq:Welch} and~\eqref{eq:orth}. The frame $\Phi$ saturates the former bound precisely when the embedding yields a simplex, and the latter bound similarly corresponds to an orthoplex.  
For this reason, the bound in~\eqref{eq:orth} is often called the \emph{orthoplex bound} or the \emph{second Rankin bound}.  
The results in~\cite{GrassPack} specifically concern packings in Euclidean space but generalize to complex space with almost no adjustment.  
We note that the approach of~\cite{GrassPack} allows one to prove variants of~\eqref{eq:Welch} and~\eqref{eq:orth} for subspaces of dimension higher than $1$.  
Another early proof of the Welch--Rankin bound appeared in the literature on equiangular lines in real Euclidean space, like the proof in~\cite{VaSei66}, which concerns eigenvalues of so-called \emph{Seidel adjacency matrices}.  
Again, the proof generalizes trivially to the complex case.  

Welch was a engineer who proved~\eqref{eq:Welch} for an application to multichannel communications.  
He actually proved a slightly more general bound, namely that for all $t \in \bN$, it holds that
\[
\mu(\Phi)^{2t} \geq \frac{1}{n-1} \left( \frac{n}{\binom{d+t-1}{t}} - 1\right).
\]
This bound is dominated by the bounds listed in Theorem~\ref{thm:bounds}.  
For $t>1$, related bounds involving the sum of the $2t$ powers of all the absolute values of the inner products end up being useful in applications, but not for the purposes of this paper (see, e.g.,~\cite{Wal17}).  
We also note that the Welch--Rankin bound was independently proven in a lighthearted article written in homage to the Gram matrix~\cite{Ros97}.  
The first appearance of the Welch--Rankin bound in the frame literature was likely~\cite{StH03}, where the authors noted that saturating the bound happens if and only if one has an equiangular tight frame.  
The Welch--Rankin bound may also be proven using linear programming techniques (see, e.g.,~\cite{HHM17}).  
Finally, we point out a new proof using a certain $\ell^1$--$\ell^\infty$ duality that appears in two recent papers~\cite{MMP19,Glaz19}.

Most known examples of frames that achieve equality in the Welch--Rankin bound arise from combinatorial designs. 
A table of these equiangular tight frames, as well as a review of many of the known constructions, may be found in~\cite{FM15}. 
As an example, one may always construct an equiangular tight frame of $d+1$ vectors in $\bF^d$ called a \emph{simplex} by centering a regular simplex at the origin, scaling the simplex so that its vertices reside in the unit sphere, and then collecting the resulting vertices as vectors in the frame.

The fact that the Welch--Rankin bound can only be saturated if $n \leq \cZ(d,\bF)$ was first noted by the authors of~\cite{LS1973}, where they attributed the bound for Euclidean space to correspondence with Gerzon.   
As above, the proof of the bound for the complex case is nearly identical to the original arguments. 
The fact that $n \leq \cZ(n-d,\bF)$ must hold comes from considering the so-called \emph{Naimark complement} of the equiangular tight frame, which must be an equiangular tight frame for $\bF^{n-d}$ (see, e.g.~\cite{STDH07}). 
There is an open conjecture, \emph{Zauner's conjecture}, in quantum information theory that there are always $d^2 = \cZ(d,\bC)$ vectors in $\bC^d$ which saturate the Welch--Rankin bound~\cite{Zauner1999,Zauner2011}.  
This is a deep problem for which solutions in only a finite number of dimensions have been found (see, e.g., \cite{FHS17,RBkSC04}). 
As a side note, there are also only a finite number of dimensions $d$ for which an equiangular tight frame of $\cZ(d,\bR)$ vectors in $\bR^d$ is known to exist; however, in contrast to the complex case, these are conjectured to be the only such configurations~\cite{Gill18}.

The bound in~\eqref{eq:BC} is called the \emph{Bukh--Cox} bound. 
It was recently proved in~\cite{BuCo18} by bounding the first moment of isotropic measures.  
An alternate proof, which applies Delsarte's linear program~\cite{delsarte1977} subsequently appeared in~\cite{MMP19,Glaz19}.  
Although the Bukh--Cox bound holds for all $n>d$, it is only the best bound on coherence in a small regime between $d+1$ and $d + O(\sqrt{d})$, as can be seen in the case $d=5$ and $\bF=\bC$, illustrated in Figure~\ref{fig:bounds}. 
As a function of $n$, the Bukh--Cox bound for $\bC^5$ is only greater than or equal to the Welch-Rankin bound for $n \in [6,7.7912\ldots]$.  
At $n=6$, the two bounds are equal and are saturated by a simplex.  
Thus, in the case of $\mathbb{C}^5$, the only $n\in\mathbb{N}$ for which the Bukh--Cox bound is the unique largest bound from Theorem~\ref{thm:bounds} is $n=7$.
For this value of $n$, the Bukh-Cox bound is $0.2645\ldots$ while the best known packing has coherence $0.2664\ldots$.

The orthoplex bound~\eqref{eq:orth} is the only bound in Theorem~\ref{thm:bounds} which does not depend on $n$.  
Thus, if one constructs $\Phi$ that saturates the orthoplex bound and has $n>\bZ(d,\bF)+1$, one may simply remove a vector to obtain a configuration of vectors which also saturates the orthplex bound.
An example of this phenomenon is visible in Figure~\ref{fig:bounds}.  
Namely, there exists a $\Phi\in \UN(5,30,\bC)$ which saturates the orthoplex bound, and so one may remove vectors one at a time to obtain Grassmannian frames of $29$, $28$, $27$, and $26$ vectors in $\bC^5$.  
The original $\Phi$ of $30=6\cdot 5$ vectors in $\bC^5$ is known as a \emph{maximal set of mutually unbiased bases}. 
Sets of mutually unbiased bases in $\bC^d$ have coherence $1/\sqrt{d}$ and maximal sets have $d(d+1) > \bZ(d,\bC)$ vectors.  
Such maximal sets of mutually unbiased bases are known to exist whenever $d$ is a prime power. 
See~\cite{KlRo04,GoRo09} and the references therein for constructions and an explanation of the application of mutually unbiased bases to quantum information theory.  
Other Grassmannian frames that saturate the orthoplex bound may be found in~\cite{BH15}.

The bound in~\eqref{eq:Lev} is called the \emph{Levenstein bound} or \emph{second Levenstein bound}, where the latter name emphasizes that the bound may be determined by applying Delsarte's linear programming bound to two special polynomials while the Welch--Rankin bound (the \textit{first} Levenstein bound) results from applying the linear programming bound to one special polynomial~\cite{HHM17}.  
The original proof of the Levenstein bound appears in~\cite{Lev82}.  
Note that in Figure~\ref{fig:bounds}, the Levenstein bound equals the Welch--Rankin bound when $n = \cZ(5,\bC)$ and surpasses the orthoplex bound at $n = 31$, which is a much smaller number of vectors than $48 = 2(\cZ(5,\bC)-1)$. 
In general, the orthoplex bound cannot be saturated for values of $n$ all the way up to the bound $2(\cZ(d,\bF)-1)$ listed in Theorem~\ref{thm:bounds}.
One may characterize all $\Phi$ which saturate~\eqref{eq:Welch} or~\eqref{eq:Lev} as so-called \emph{projective $t$-designs} with particular angle sets~\cite{HHM17,Wal17}.

The linear programming approach to proving bounds on spherical codes was developed by Delsarte and coauthors~\cite{delsarte1977} based on similar work by Delsarte on codes with a finite alphabet~\cite{Del72,Del73} and generalized by Kabatyansky and Levenshtein~\cite{KaLe78}. 
Such an approach can be used to prove the Bukh--Cox, Welch--Rankin, and Levenstein bounds.  
Three-point bounds via semidefinite programming would likely produce even better bounds~\cite{BV2008,dLV15}. 

There has been considerably less work in proving the optimality of configurations that do not saturate one of the bounds in Theorem~\ref{thm:bounds}.
Notice that optimally packing points in $\bR\bP^1\cong S^1$ is a simple exercise in the pigeonhole principle~\cite{BKo06}.  
By contrast, optimally packing points in $\bC\bP^1\cong S^2$ is equivalent to the (hard) Tammes problem, and to date, optimality has only been characterized for $n\leq 14$ and $n=24$ (see~\cite{toth1949densest,robinson1961arrangement,schutte1951kugel,danzer1986finite,musin2012strong,musin2015tammes}).
Most of these proofs exploit properties of $S^2$ to isolate necessary combinatorial features of the \textit{contact graph} of the optimal spherical code, that is, the graph whose vertex set is the set of points being packed, with two points sharing an edge if their distance equals the minimum distance of the packing.
These features can be used to isolate several contact graph candidates, most of which can then be ruled out by (spherical) geometric arguments.
Similar techniques have been leveraged to find the optimal packing of $n=5$ points in $\bR\bP^2$ (see~\cite{Toth65,BKo06}), of $n=8$ points in $\bR\bP^2$ (see~\cite{MiPa19}), and of $n=6$ points in $\bR\bP^3$ (see~\cite{FJM18}); in this last case, techniques from spherical geometry are no longer applicable, and so the authors resorted to techniques from algebraic geometry.


\section{Game of Sloanes and Website}\label{sec:games}

The primary goal of this paper is to announce the new website~\cite{GameofSloanes}, which will be live by the SPIE Optics and Photonics Conference.
The website will be a repository for the best known packings in complex projective space, much like Neil Sloane's site~\cite{Sloane} is the standard for packings in real projective space.    
The putatively optimal packings will be listed like a leaderboard, gamifying the hunt for optimal complex Grassmannian frames.  
In honor of Neil Sloane, we are calling this game the \textit{Game of Sloanes}.

To initiate the Game of Sloanes, we scraped Neil Sloane's website on real line packings~\cite{Sloane} (considering these are also complex line packings) and his website on spherical codes~\cite{Sloane2} (as these yield packings in $\bC\bP^1$).
We then used a handful of algorithms to systematically find better packings for parameters in the range\footnote{We are also interested in packings beyond this range, so please submit those as well!} $3 \leq d \leq 7$ and $d+2 \leq n \leq 49$.
What seemed to work best was to repeatedly run a slightly modified version of the Matlab code found on~\cite{CodebookSite}, which implements the algorithm in~\cite{MeDa14} that is based on sequential smooth optimization on the Grassmannian manifold and the use of smooth penalty functions.
The resulting vector configurations were then tweaked by simple gradient descent or an alternating projection algorithm related to the approach in~\cite{TDHS05}.
We also tinkered with the phases in the Gram matrix in an attempt to limit them to a small set.

In addition to listing the best known coherence for each $d$ and $n$, the leaderboard also contains the maximum of the bounds in Theorem~\ref{thm:bounds} corresponding to those parameters; this is provided in the column titled ``Lower Bound.'' 
The column named ``Creator Notes'' contains one of three types of information:
\begin{itemize}
\item
the initials of the person who submitted the packing (e.g., \texttt{ejk} for Emily J.\ King),
\item
a note that indicates the construction (e.g., \texttt{etf} for equiangular tight frame), or
\item
\texttt{AUTO}, which indicates that the best known packing can be obtained by removing a vector from the best known packing in the same dimension but with one more vector.
\end{itemize}
For the reader's benefit, the website will also provide the Matlab code we used to compute numerical packings, as well as the packings themselves (as Matlab files).

There has been some success in leveraging certain mathematical structure to ``exactify'' putatively optimal packings that were numerically calculated~\cite{ACFW18,FJM18,MiPa19}.  
It is our hope that once more numerical packings are publicly available, more exact constructions will be found.  
As an example, we propose two putatively optimal configurations that enjoy closed-form descriptions.

\begin{conj}
The frames of $8$ vectors in $\bC^3$ that result from removing a single vector from any of the known equiangular tight frames of $9$ vectors in $\bC^3$ (see, e.g,~\cite{DBBA2013}) are Grassmannian frames.
\end{conj}

We note that it has been shown via Gr\"obner basis calculations that $8$ vectors in $\bC^3$ cannot saturate the Welch--Rankin bound~\cite{Szo14}.  
This is the only proof of non-existence of a complex equiangular tight frame that does not leverage the bound $n \leq \min\{ \cZ(d,\bC) , \cZ(n-d,\bC)\}$.

\begin{conj}
The following is a Grassmannian frame of $5$ vectors in $\mathbb{C}^3$:
\[
\Phi = \left[ \begin{array}{ccccc} a  & b & b & c & c \\  b  & a & b & cw & cw^2 \\   b  & b & a & cw^2 & cw \end{array} \right],
\quad
a = \frac{\sqrt{13}+\sqrt{2+\sqrt{13}}-1}{3\sqrt{3}},
\quad
b = \sqrt{\frac{1-a^2}{2}},
\quad
c = \frac{1}{\sqrt{3}},
\quad
w=e^{2\pi i/3}.
\]
\end{conj}

In the future, we would also like to post putatively optimal packings for related packing problems.
Packings of real subspaces of dimension greater than $1$ can be found on Sloane's original website~\cite{Sloane}.  
The website~\cite{SomeCompSite} also has a few packings of subspaces of various dimensions in complex space; however, except for the known equiangular tight frames, our line packings improve on the configurations on that site.  
The algorithm in~\cite{MeDa14} already includes the general subspace case, and the algorithm in~\cite{TDHS05} was generalized to handle subspace packings in~\cite{DHST08}.  
There are known generalizations of coherence, and for lower bounds, there are natural generalizations of the Welch--Rankin and orthoplex bounds to higher-dimensional subspaces.

\begin{open}
Generalize the Levenstein and Bukh--Cox bounds to general subspace packings.
\end{open}

Finally, $1$-Grassmannian frames are finite unit norm tight frames with the best possible coherence~\cite{HaCa17}:
\[
\Phi \in \argmin\left\{ \mu(\Psi) \, : \, \Psi \in \FUNTF(d,n,\bF) \right\}.
\]
Sometimes the $\Phi$ which satisfy~\eqref{eq:Grassman} are tight and therefore $1$-Grassmannian frames.  
For example, equiangular tight frames and maximal sets of mutually unbiased bases are two examples of tight Grassmannian frames.  
The benefit of using a $1$-Grassmannian frame is that one has maximum geometric spread while still being able to compute a change of basis \`a la~\eqref{eq:tight}.

We look forward to your Game of Sloanes submissions and would appreciate your feedback!

\section*{ACKNOWLEDGMENTS}
The authors are indebted to Matt Fickus, who wrote code to scrape packing data from Neil Sloane's website.
This work was supported in part by AFOSR FA9550-18-1-0107, NSF DMS 1829955, NSF DMS 1830066, and the Simons Institute of the Theory of Computing.


\begin{thebibliography}{10}

\bibitem{BFMW13}
Bandeira, A.~S., Fickus, M., Mixon, D.~G., and Wong, P., ``The road to
  deterministic matrices with the restricted isometry property,'' {\em J.
  Fourier Anal. Appl.}~{\bf 19}(6),  1123--1149 (2013).

\bibitem{MQKF13}
Mixon, D.~G., Quinn, C.~J., Kiyavash, N., and Fickus, M., ``Fingerprinting with
  equiangular tight frames,'' {\em IEEE Trans. Inform. Theory}~{\bf 59}(3),
  1855--1865 (2013).

\bibitem{RBkSC04}
Renes, J.~M., Blume-Kohout, R., Scott, A.~J., and Caves, C.~M., ``Symmetric
  informationally complete quantum measurements,'' {\em J. Math. Phys.}~{\bf
  45}(6),  2171--2180 (2004).

\bibitem{FHS17}
Fuchs, C.~A., Hoang, M.~C., and Stacey, B.~C., ``The {SIC} question: History
  and state of play,'' {\em Axioms}~{\bf 6}(3) (2017).

\bibitem{StH03}
Strohmer, T. and Heath, Jr., R.~W., ``Grassmannian frames with applications to
  coding and communication,'' {\em Appl. Comput. Harmon. Anal.}~{\bf 14}(3),
  257--275 (2003).

\bibitem{MeDa14}
Medra, A. and Davidson, T.~N., ``Flexible codebook design for limited feedback
  systems via sequential smooth optimization on the {G}rassmannian manifold,''
  {\em IEEE Trans. Signal Process.}~{\bf 62}(5),  1305--1318 (2014).

\bibitem{Welch}
{Welch}, L., ``Lower bounds on the maximum cross correlation of signals,'' {\em
  IEEE Trans Inf Theory}~{\bf 20},  397--399 (May 1974).

\bibitem{Toth65}
Fejes~T\'{o}th, L., ``Distribution of points in the elliptic plane,'' {\em Acta
  Math. Acad. Sci. Hungar.}~{\bf 16},  437--440 (1965).

\bibitem{Tam30}
Tammes, P., {\em On the origin of number and arrangement of the places of exit
  on the surface of pollen-grains}, PhD thesis (1930).
\newblock Relation: https://www.rug.nl/ Rights: De Bussy.

\bibitem{GrassPack}
Conway, J.~H., Hardin, R.~H., and Sloane, N. J.~A., ``Packing lines, planes,
  etc.: packings in {G}rassmannian spaces,'' {\em Experiment. Math.}~{\bf
  5}(2),  139--159 (1996).

\bibitem{Bod07}
Bodmann, B.~G., ``Optimal linear transmission by loss-insensitive packet
  encoding,'' {\em Appl. Comput. Harmon. Anal.}~{\bf 22}(3),  274--285 (2007).

\bibitem{Sloane}
Sloane, N.~J.~A., ``N. {J}. {A}. {S}loane: How to pack lines, planes,
  $3$-spaces, etc..'' \url{http://neilsloane.com/grass/}.

\bibitem{GameofSloanes}
Jasper, J., King, E.~J., and Mixon, D.~G., ``Game of sloanes.''
  \url{https://www.math.colostate.edu/~king/GameofSloanes.html}.

\bibitem{delsarte1977}
Delsarte, P., Goethals, J.~M., and Seidel, J.~J., ``Spherical codes and
  designs,'' {\em Geometriae Dedicata}~{\bf 6}(3),  363--388 (1977).

\bibitem{musin2009}
Musin, O.~R., ``Spherical two-distance sets,'' {\em J. Combin. Theory Ser.
  A}~{\bf 116}(4),  988--995 (2009).

\bibitem{BKo06}
Benedetto, J.~J. and Kolesar, J.~D., ``Geometric properties of {G}rassmannian
  frames for $\mathbb{R}^2$ and $\mathbb{R}^3$,'' {\em EURASIP J. Appl. Signal
  Process.}~{\bf 49850} (2006).

\bibitem{King15}
King, E.~J., ``Algebraic and geometric spread in finite frames,'' in [{\em
  Wavelets and Sparsity XVI}{\nolinebreak\hspace{0.1em}]},  Papadakis, M. and
  Goyal, V.~K., eds., {\em Society of Photo-Optical Instrumentation Engineers
  (SPIE) Conference Series} {\bf 9597} (2015).

\bibitem{Toth49}
Fejes~T\'{o}th, L., ``On the densest packing of spherical caps,'' {\em Amer.
  Math. Monthly}~{\bf 56},  330--331 (1949).

\bibitem{ConwaySloane98}
Conway, J.~H. and Sloane, N. J.~A.,  [{\em Sphere packings, lattices and
  groups}{\nolinebreak\hspace{0.1em}]}, vol.~290 of {\em Grundlehren der
  Mathematischen Wissenschaften [Fundamental Principles of Mathematical
  Sciences]}, Springer-Verlag, New York, third~ed. (1999).
\newblock With additional contributions by E. Bannai, R. E. Borcherds, J.
  Leech, S. P. Norton, A. M. Odlyzko, R. A. Parker, L. Queen and B. B. Venkov.

\bibitem{FJM18}
Fickus, M., Jasper, J., and Mixon, D.~G., ``Packings in real projective
  spaces,'' {\em SIAM J. Appl. Algebra Geom.}~{\bf 2}(3),  377--409 (2018).

\bibitem{Glaz19}
Glazyrin, A., ``Moments of isotropic measures and optimal projective codes,''
  arXiv:1904.11159 (2019).

\bibitem{VaSei66}
Van~Lint, J.~H. and Seidel, J.~J., ``Equilateral point sets in elliptic
  geometry,'' {\em Indagationes Mathematicae, Proc. Koninkl. Ned. Akad.
  Wetenschap. Ser. A}~{\bf 69}(3),  335--34 (1966).

\bibitem{HaCa17}
{Haas IV}, J.~I. and Casazza, P.~G., ``On the structures of {G}rassmannian
  frames,'' in [{\em 2017 International Conference on Sampling Theory and
  Applications (SampTA)}{\nolinebreak\hspace{0.1em}]},   377--380 (July 2017).

\bibitem{Ran55}
Rankin, R.~A., ``The closest packing of spherical caps in $n$ dimensions,''
  {\em Proc Glasgow Math Assoc}~{\bf 2},  139--144 (1955).

\bibitem{Wal17}
Waldron, S., ``A sharpening of the {W}elch bounds and the existence of real and
  complex spherical {$t$}-designs,'' {\em IEEE Trans. Inform. Theory}~{\bf
  63}(11),  6849--6857 (2017).

\bibitem{Ros97}
Rosenfeld, M., ``In praise of the {G}ram matrix,'' in [{\em The mathematics of
  {P}aul {E}rd\H{o}s, {II}}{\nolinebreak\hspace{0.1em}]},  {\em Algorithms
  Combin.} {\bf 14},  318--323, Springer, Berlin (1997).

\bibitem{HHM17}
{Haas IV}, J.~I., Hammen, N., and Mixon, D.~G., ``The {L}evenstein bound for
  packings in projective spaces,'' in [{\em Wavelets and Sparsity
  XVII}{\nolinebreak\hspace{0.1em}]},  Lu, Y.~M., Ville, D. V.~D., and
  Papadakis, M., eds., {\em Society of Photo-Optical Instrumentation Engineers
  (SPIE) Conference Series} {\bf 10394} (2017).

\bibitem{MMP19}
Magsino, M., Mixon, D.~G., and Parshall, H., ``A {D}elsarte-style proof of the
  {B}ukh--{C}ox bound,'' {\em Sampl. Theory Signal Image Process.}~{\bf Special
  Issue SampTA 2019} (2019).
\newblock To appear.

\bibitem{FM15}
Fickus, M. and Mixon, D.~G., ``Tables of the existence of equiangular tight
  frames,'' arXiv:1504.00253 (2016).

\bibitem{LS1973}
Lemmens, P. W.~H. and Seidel, J.~J., ``Equiangular lines,'' {\em J.
  Algebra}~{\bf 24},  494--512 (1973).

\bibitem{STDH07}
Sustik, M.~A., Tropp, J.~A., Dhillon, I.~S., and Jr, R. W.~H., ``On the
  existence of equiangular tight frames,'' {\em Linear Algebra Appl.}~{\bf
  426},  619--635 (2007).

\bibitem{Zauner1999}
Zauner, G., {\em Quantendesigns - {G}rundz{\"u}ge einer nichtkommutativen
  {D}esigntheorie}, PhD thesis, University Wien (Austria) (1999).
\newblock English translation in International Journal of Quantum Information
  (IJQI) 9 (1), 445--507, 2011.

\bibitem{Zauner2011}
Zauner, G., ``Quantum designs: foundations of a noncommutative design theory,''
  {\em Int. J. Quantum Inf.}~{\bf 9}(1),  445--507 (2011).

\bibitem{Gill18}
Gillespie, N.~I., ``Equiangular lines, incoherent sets and quasi-symmetric
  designs,'' arXiv:1809.05739 (2018).

\bibitem{BuCo18}
Bukh, B. and Cox, C., ``Nearly orthogonal vectors and small antipodal spherical
  codes,'' arXiv:1803.02949 (2018).

\bibitem{KlRo04}
Klappenecker, A. and R\"{o}tteler, M., ``Constructions of mutually unbiased
  bases,'' in [{\em Finite fields and
  applications}{\nolinebreak\hspace{0.1em}]},  {\em Lecture Notes in Comput.
  Sci.} {\bf 2948},  137--144, Springer, Berlin (2004).

\bibitem{GoRo09}
Godsil, C. and Roy, A., ``Equiangular lines, mutually unbiased bases, and spin
  models,'' {\em European J. Combin.}~{\bf 30}(1),  246--262 (2009).

\bibitem{BH15}
Bodmann, B.~G. and Haas, J., ``Achieving the orthoplex bound and constructing
  weighted complex projective 2-designs with {S}inger sets,'' {\em Linear
  Algebra Appl.}~{\bf 511},  54--71 (2016).

\bibitem{Lev82}
Leven\v{s}te\u{\i}n, V.~I., ``Bounds on the maximal cardinality of a code with
  bounded modules of the inner product,'' {\em Soviet Math. Dokl.}~{\bf 25},
  526--531 (1982).

\bibitem{Del72}
Delsarte, P., ``Bounds for unrestricted codes, by linear programming,'' {\em
  Philips Res. Rep.}~{\bf 27},  272--289 (1972).

\bibitem{Del73}
Delsarte, P., ``An algebraic approach to the association schemes of coding
  theory,'' {\em Philips Res. Rep. Suppl.} (10),  vi+97 (1973).

\bibitem{KaLe78}
Kabatjanski\u{\i}, G.~A. and Leven\v{s}te\u{\i}n, V.~I., ``Bounds for packings
  on the sphere and in space,'' {\em Problemy Pereda\v{c}i Informacii}~{\bf
  14}(1),  3--25 (1978).

\bibitem{BV2008}
Bachoc, C. and Vallentin, F., ``New upper bounds for kissing numbers from
  semidefinite programming,'' {\em J. Amer. Math. Soc.}~{\bf 21},  909--924
  (2008).

\bibitem{dLV15}
de~Laat, D. and Vallentin, F., ``A semidefinite programming hierarchy for
  packing problems in discrete geometry,'' {\em Math. Program.}~{\bf 151}(2,
  Ser. B),  529--553 (2015).

\bibitem{toth1949densest}
Fejes~T{\'o}th, L., ``On the densest packing of spherical caps,'' {\em The
  American Mathematical Monthly}~{\bf 56}(5),  330--331 (1949).

\bibitem{robinson1961arrangement}
Robinson, R.~M., ``Arrangement of 24 points on a sphere,'' {\em Mathematische
  Annalen}~{\bf 144}(1),  17--48 (1961).

\bibitem{schutte1951kugel}
Sch{\"u}tte, K. and Van~der Waerden, B., ``Auf welcher kugel haben 5, 6, 7, 8
  oder 9 punkte mit mindestabstand eins platz?,'' {\em Mathematische
  Annalen}~{\bf 123}(1),  96--124 (1951).

\bibitem{danzer1986finite}
Danzer, L., ``Finite point-sets on ${S}^2$ with minimum distance as large as
  possible,'' {\em Discrete mathematics}~{\bf 60},  3--66 (1986).

\bibitem{musin2012strong}
Musin, O.~R. and Tarasov, A.~S., ``The strong thirteen spheres problem,'' {\em
  Discrete \& Computational Geometry}~{\bf 48}(1),  128--141 (2012).

\bibitem{musin2015tammes}
Musin, O.~R. and Tarasov, A.~S., ``The {T}ammes problem for ${N}=14$,'' {\em
  Experimental Mathematics}~{\bf 24}(4),  460--468 (2015).

\bibitem{MiPa19}
Mixon, D.~G. and Parshall, H., ``Exact line packings from numerical
  solutions,'' {\em Sampl. Theory Signal Image Process.}~{\bf Special Issue
  SampTA 2019} (2019).
\newblock To appear.

\bibitem{Sloane2}
Sloane, N.~J.~A., ``Spherical codes.'' \url{http://neilsloane.com/packings/}.

\bibitem{CodebookSite}
Medra, A. and Davidson, T.~N., ``Flexible codebook design for limited feedback
  systems.''
  \url{http://www.ece.mcmaster.ca/~davidson/pubs/Flexible_codebook_design.html}.

\bibitem{TDHS05}
Tropp, J.~A., Dhillon, I.~S., Jr., R.~W.~H., and Strohmer, T., ``Designing
  structured tight frames via alternating projection,'' {\em IEEE Trans. Info.
  Theory}~{\bf 51}(1),  188--209 (2005).

\bibitem{ACFW18}
Appleby, M., Chien, T.-Y., Flammia, S., and Waldron, S., ``Constructing exact
  symmetric informationally complete measurements from numerical solutions,''
  {\em J. Phys. A}~{\bf 51}(16) (2018).

\bibitem{DBBA2013}
Dang, H.~B., Blanchfield, K., Bengtsson, I., and Appleby, D.~M., ``Linear
  dependencies in {W}eyl--{H}eisenberg orbits,'' {\em Quantum Information
  Processing}~{\bf 12},  3449--3475 (Nov 2013).

\bibitem{Szo14}
Sz{\"o}ll{\H{o}}si, F., ``All complex equiangular tight frames in dimension
  $3$,'' arXiv:1402.6429 (2014).

\bibitem{SomeCompSite}
Love, D.~J., ``Grassmannian subspace packing.''
  \url{https://engineering.purdue.edu/\~djlove/grass.html}.

\bibitem{DHST08}
Dhillon, I.~S., Heath, Jr., R.~W., Strohmer, T., and Tropp, J.~A.,
  ``Constructing packings in {G}rassmannian manifolds via alternating
  projection,'' {\em Experiment. Math.}~{\bf 17}(1),  9--35 (2008).

\end{thebibliography}

\end{document}